\documentclass[a4paper,oneside,11pt]{article}

\usepackage{amsmath,amsfonts,amscd,amssymb}
\usepackage{chngcntr}
\usepackage{longtable,geometry}
\usepackage[english]{babel}
\usepackage[utf8]{inputenc}
\usepackage[active]{srcltx}
\usepackage[T1]{fontenc}
\usepackage{graphicx}
\usepackage{pstricks}
\usepackage{bbm}
\usepackage{mathtools}
\usepackage{hyperref}
\usepackage{commath}
\usepackage{MnSymbol}
\usepackage{stmaryrd}
\usepackage{nicefrac}
\usepackage{calrsfs}
\usepackage{relsize}
\usepackage{cases}
\usepackage{enumitem}

\usepackage{xcolor}
\usepackage{framed}

\colorlet{shadecolor}{blue!15}

\usepackage{amsthm}
\newtheorem{theorem}{Theorem}

\newtheorem{corollary}[theorem]{Corollary}
\newtheorem{lemma}[theorem]{Lemma}

\newtheorem{definition}[theorem]{Definition}

\newtheorem{remark}[theorem]{Remark}

\newcommand{\bbE}{\mathbb{E}}

\newcommand{\bbP}{\mathbb{P}}

\newcommand{\bbR}{\mathbb{R}}

\newcommand{\bbZ}{\mathbb{Z}}

\newcommand{\ep}{\varepsilon}

\newcommand{\eps}{\varepsilon}
\newcommand{\var}{\mathrm{Var}}
\newcommand{\infl}{\mathrm{Inf}}

\numberwithin{equation}{section}

\newcommand{\rk}[1]{\bgroup\color{red}%
  \par\medskip\hrule\smallskip%
  \noindent\textbf{#1}%
  \par\smallskip\hrule\medskip\egroup}

\counterwithout{equation}{section}
\title{Exponential decay of connection probabilities for subcritical Voronoi percolation in $\bbR^d$}
\author{Hugo Duminil-Copin\thanks{Universit\'e de Gen\`eve} \thanks{Institut des Hautes \'Etudes Scientifiques} , Aran Raoufi\addtocounter{footnote}{-1}\footnotemark\ , Vincent Tassion\thanks{ETH Zurich}}
\date{\today}

\begin{document}
\maketitle
 
 \begin{abstract}
 We prove that for Voronoi percolation on $\bbR^d$, there exists $p_c\in[0,1]$ such that 
 \begin{itemize}
 \item for $p<p_c$, there exists $c_p>0$ such that $\bbP_p[0\text{ connected to distance }n]\le \exp(-c_pn)$,
 \item there exists $c>0$ such that for $p>p_c$, $\bbP_p[0\text{ connected to }\infty]\ge c(p-p_c)$.
 \end{itemize}
For dimension 2, this result offers a new way of showing that $p_c(2)=1/2$.
This paper belongs to a series of papers using the theory of algorithms to prove sharpness of the phase transition; see \cite{DumRaoTas16a,DumRaoTas16c}.
 \end{abstract}
\section{Introduction}

\paragraph{Motivation.} Bernoulli percolation was introduced in \cite{BroHam57} by Broadbent and Hammersley to model the diffusion of a liquid in a porous medium. Originally defined on a lattice, the model was later generalized to a number of other contexts. Of particular interest is the developments of percolation in continuum environment, see \cite{MeeRoy08} for a book on the subject. 

One of the most classical such model is provided by Voronoi percolation, where the Voronoi cells associated to a Poisson point process in $\mathbb R^d$ are colored independently black or white with respective probability $p$ and $1-p$. Voronoi percolation behaves very similarly to Bernoulli percolation, but is harder to study, due to local dependencies (the colors of two disjoint points are always correlated, since two points have always a positive probability to belong to the same cell). Because of these dependencies, several techniques for Bernoulli percolation do not apply, and the anylsis of Voronoi percolation requires to develop new and more robust methods. In the celebrated work \cite{BolRio06}, Bollobás and Riordan proved that Voronoi percolation in the plane undergoes a sharp phase transition at the critical parameter $p=1/2$, meaning that for $p>1/2$, the connected component of black cells containing 0 is infinite with positive probability, while for $p<1/2$, it has probability of having radius larger than $n$ decaying exponentially fast in $n$. Since this result, several other results came to complement the picture on planar Voronoi percolation, including a fine description of the critical behavior \cite{ahlgritas2016,tassion2016crossing}. The recent advances in the understanding of Voronoi percolation were mostly restricted to the planar case, and several fundamental questions, including sharpness of the phase transition, remained widely open in higher dimension. This article provides a first proof of sharpness for Voronoi percolation in any dimension $d\ge2$. As a consequence, it also offers an alternative computation of the critical point in the two-dimensional case.

\bigbreak Let $d\geq 2$ be a positive integer and let $\bbR^d$ be the $d$-dimensional Euclidean space with $\|\cdot\|$ denote the $\ell^2$ norm. For $r>0$, set $\mathsf B_r:=\{y\in\bbR^d:\|y\|\le r\}$ and $\mathsf S_r:=\{y\in\bbR^d:\|y\|= r\}$ for the ball and sphere of radius $r$ around the origin.

Let $\bbP_p$ denote the Voronoi percolation measure with parameter $p$ on $\bbR^d$, that is $\bbP_p$ is the law of two independent point processes $\eta^b$ and $\eta^w$ with respective intensities $p$ and $1-p$ (here, $\eta^b$ and $\eta^w$ are two locally finite subsets of $\mathbb R^d$). Define $\eta = \eta^b \cup \eta^w$. For a point $x \in \eta$, define the {\em Voronoi cell} of $x$
$$ C(x) := \big\{ \, y \in \bbR^d:  \|x - y\| = \min_ {x' \in \eta}  \| x'- y\|  \,\big\}.$$
The measure $\bbP_p$ induces a coloring $\omega$ on the points of $\bbR^d$ defined as follows. 
Set $\omega(y) = 1$ for every $y$ belonging to the Voronoi cell of some $x\in \eta^b$. Set $\omega(y)=0$ for all the other points in $\bbR^d$. We say that $y$ is {\em black} if $\omega(y)=1$, and {\em white} otherwise.

For $x, y \in \bbR^d$, let the event $x$ connected to $y$ (denoted by $\{ x \longleftrightarrow y\}$) be the existence of a continuous path of black points connecting $x$ to $y$. If $X, Y \subset \bbR^d$, the event $\{X \longleftrightarrow Y\}$ denotes existence of $x \in X$ and $y \in Y$ such that $x$ is connected to $y$. Also, $\{0 \longleftrightarrow \infty \}$ is the event that $0$ belongs to an unbounded connected component of black points. For $p\in[0,1]$ and $n\ge0$, define 
$\theta(p) := \bbP_p[0 \longleftrightarrow \infty]$ and 
$\theta_n(p) := \bbP_p[0 \longleftrightarrow \mathsf S_n]$.
Finally, we set $p_c:=\inf\{p\in[0,1]:\theta(p)>0\}.$

 The main result of this paper is the following theorem.

\begin{theorem} \label{thm:main}Fix $d\ge2$. 
For any $p<p_c$, there exists $c_p>0$ such that for any $n \geq 1$,
\begin{equation}\label{eq:main}
 \theta_n(p)\leq  \exp({-c_p n}).
\end{equation}
Furthermore, there exists $c>0$ such that 
$
\theta(p) \geq c(p-p_c)
$ for any $p>p_c$.
\end{theorem}
This result has an immediate corollary, namely the result of Bollob\'as and Riordan \cite{BolRio06} on planar Voronoi percolation.
\begin{corollary}\label{cor:main}
The critical parameter of Voronoi percolation on $\bbR^2$ is equal to $1/2$. Furthermore, $\theta(1/2)=0$.
\end{corollary}
Existing proofs of exponential decay for more standard models such as Bernoulli percolation \cite{Men86,AizBar87,DumTas15} or the Ising model \cite{AizBarFer87,DumTas15} do not extend to the context of Voronoi percolation. The reason is a lack of BK-type inequality.  In two dimensions, Bollob\`as and Riordan use crossing probabilities and introduce tools from Boolean functions \cite{friedgut1996every} to bypass this difficulty. This strategy was proved very fruitful in two dimension, since several results were proved for dependent percolation models using similar ideas; see e.g.~\cite{BefDum12,DumRaoTas16}. Unfortunately, applying such arguments in higher dimension seemed to be very challenging, so that even Bernoulli-type percolation models remained out of reach of the previous method. Recently, a new technique based on randomized algorithms was introduced to prove sharpness of the phase transition for the random-cluster and Potts models on transitive graphs \cite{DumRaoTas16a}. 
This method, based on an inequality connecting randomized algorithms and influences in a product space first proved in \cite{OSSS}, seems applicable to a variety of continuum models including Voronoi percolation or Boolean percolation \cite{DumRaoTas16c}.

The strategy consists in proving a family of differential inequalities. More precisely, fix $\delta>0$ such that $p_c\in(\delta,1-\delta)$. We will prove that there exists $c>0$ such that for all $n\geq 1$ and $p\in[\delta,1-\delta]$,
\begin{equation} \label{eq:mlem}
{\theta'_n(p)} ~\geq~ c \,  {\frac{n}{{S_n(p)}}} \, \theta_n(p),
\end{equation} 
where $S_n:= \sum_{k=0}^{n-1} \theta_k$.

The proof of Theorem~\ref{thm:main} follows from \eqref{eq:mlem} by applying the following lemma to $f_n=\theta_n /c$.   This lemma can be found in \cite{DumRaoTas16a}.
  \begin{lemma}\label{lem:technical}
Consider a converging sequence of increasing differentiable functions $f_n:[\alpha_0,\alpha_1]\longrightarrow [0,M]$ satisfying
 \begin{equation}\label{eq:technical}f_n'\ge \frac{n}{\Sigma_{n}}f_n\end{equation}
 for all $n\ge1$, where $\Sigma_n=\sum_{k=0}^{n-1}f_k$. Then, there exists $\beta \in[\alpha_0,\alpha_1]$ such that  
 \begin{itemize}
 \item For any $\beta<\beta_1$, there exists $c_\beta>0$ such that for any $n$ large enough,
 $f_n(\beta)\le M\exp(-c_\beta n).$
  \item For any $\beta>\beta_1$, $\displaystyle f=\lim_{n\rightarrow \infty}f_n$ satisfies $f(\beta)\ge \beta-\beta_1.$
 \end{itemize}
  \end{lemma}

The paper is organized as follows. The next section contains some preliminaries. In Section 3, we prove \eqref{eq:mlem}. Section 4 contains the proof of Corollary~\ref{cor:main}. For completeness, we include the proof of Lemma~\ref{lem:technical} in Section 5.

\section{Preliminaries}

\subsection{Monotone events and the  FKG inequality}
\label{sec:fkg-inequality}
An event $A$ is  said to be \emph{increasing} if for every configurations $(\eta^b,\eta^w)$, $(\bar\eta^b,\bar\eta^w)$,
\begin{equation*}
  \left.\begin{array}[c]{c}
    (\eta^b,\eta^w)\in A\\
    \eta^b\subset \bar\eta^b, \,  \eta^w\supset \bar\eta^w
  \end{array}
\right\}\implies   (\bar\eta^b,\bar\eta^w)\in A.
\end{equation*}
An event is said to be \emph{decreasing} if its complement is increasing. 
The FKG inequality for Voronoi percolation (see e.g.\@
 \cite{BolRio06b}) states that for any increasing events $A$ and $B$,
\begin{equation}
  \label{eq:5}
  \bbP_p[A\cap B]\ge\bbP_p[A]\bbP_p[B].\tag{FKG}
\end{equation}
Note that it implies that $\bbP_p[A\cap B]\le\bbP_p[A]\bbP_p[B]$ whenever $A$ is increasing and $B$ is decreasing.

\subsection{A Russo's type formula for Voronoi percolation}
\label{sec:russ-form-voron}

For an increasing event $A$, define the set of {\em pivotal points} 
$$\mathsf{Piv}_A := \left \lbrace x \in \eta: {\bf 1}_A ( \eta^b \setminus \{x\}, \eta^w \cup  \{x\})\neq  {\bf 1}_A ( \eta^b \cup  \{x\}, \eta^w \setminus \{x\}) \right \rbrace.$$
Call an increasing event $A$ {\em local} if there exists $n\ge0$ such that $A$ is measurable with respect to the $\sigma$-algebra generated by $\{\omega(x)\}_{x\in \mathsf B_n}$.
\begin{lemma}\label{Russo}
Consider a local increasing event $A$. Then, $p\mapsto \bbP_p [A]$ is differentiable and
\begin{align*}
 \frac{{\rm d}\bbP_p [A]}{{\rm d}p}  = \bbE_p [|\mathsf{Piv}_A|] .
\end{align*}
\end{lemma}
Note that even though the event $A$ may depend only on the colors of the points in $\mathsf B_n$, the set $\mathsf{Piv}_A$ can a priori contain points outside the ball. Nonetheless, it is simple to check that $|\mathsf{Piv}_A|$ is integrable. Indeed, we have
\begin{equation}
  \label{eq:6}
  |\mathsf{Piv}_A|\le |D_n(\eta)| 
\end{equation}
where $D_n(\eta)$ is the set of points in $\eta$, whose cells intersect the ball $\mathsf B_n$. The integrability of $D_n$ follows from standard estimates of the Poisson-Voronoi tessellation. For example, observe that there exists $c>0$ such that for every $t\ge n$, $\bbP_p[D_n \cap\mathsf (\bbR^d\setminus B_{4t})\neq\emptyset]\le \bbP_p[D_n \cap B_{t} =\emptyset]\le e^{-ct^d}$ and $\bbP_p[|D_n| \cap B_{4t}\geq t^{d+1}]\le \bbP_p[\eta \cap B_{4t}\geq t^{d+1}] \le e^{-ct}$.

\begin{proof}[Proof of Lemma~\ref{Russo}] 
In this proof, $d\eta$ denotes the law of $\eta$ (in particular it does not contain information on colors). Write
$$ \bbP_{p+\delta} [A] - \bbP_{p} [A] = \int_{\eta} \bbP_{p+\delta}[A \,|\, \eta] -\bbP_p[A \, | \, \eta] \, \, d\eta.$$
Condition on $\eta$, the law of $\eta^b$ is Bernoulli percolation with parameter $p$ on points of $\eta$. 
Since $A$ is measurable with respect to the $\sigma$-algebra generated by $\{\omega(x)\}_{x \in \mathsf B_n}$, apply Russo's formula (for Bernoulli percolation) and Fubini to get
\begin{align*}
\bbP_{p+\delta} \big[A \big] - \bbP_{p} \big[A\big] &= \int_{\eta} \Big(\int_{p\leq s \leq p+\delta} \bbE_{s} [|\mathsf{Piv}_A|~| \eta ]ds\Big)\,d\eta \\
&= \int_{p\leq s \leq p+\delta} \Big(\int_{\eta}  \bbE_{s} [|\mathsf{Piv}_A|~| \eta ]d\eta\Big)\,ds\\
&= \int_{p\leq s \leq p+\delta} \bbE_{s} [|\mathsf{Piv}_A|]\,ds.
\end{align*}
The proof follows by continuity in $s$ of $\bbE_{s} [|\mathsf{Piv}_A|]$,which is direct consequence of the domination~\eqref{eq:6}.
\end{proof}

\subsection{The OSSS inequality}
\label{sec:osss-ineq-bool}

Assume $ I$ is a countable set, and let $(\Omega^I, \pi^{\otimes I})$ be a product probability space, and $f: \Omega^I \rightarrow \{0,1\}$.
An {\em algorithm $\mathsf T$} determining $f$ takes a configuration $\omega=(\omega_i)_{i\in I}\in\Omega^I$ as an input, and reveals the value of $\omega$ in different edges one by one. 
At each step, which coordinate will be revealed next depends on the values of $\omega$ revealed so far. The algorithm stops as soon as the value of $f$ is the same no matter the values of $\omega$ on the remaining coordinates. 
Here, we always assume that the algorithm stops in finite time almost surely.
We will use the following inequality. For any function $f: \Omega^I \rightarrow \{0,1\}$, and any algorithm $\mathsf T$ determining $f$, 
\begin{equation} \label{eq:OSSS}\tag{OSSS}
\var (f) \leq \sum_{i\in I} \delta_{i}(\mathsf T) \, \infl_{i}(f),
\end{equation}
where $\delta_i(\mathsf T)$ and $\infl_{i}(f)$ are respectively the {\em revealment} and the {\em influence} of the $i$-th coordinate defined by
\begin{align*}\delta_i(\mathsf T) &:=  \pi^{\otimes I}[\text{$\mathsf T$ reveals the value of } \omega_i],\\
\infl_i(f) &:= \pi^{\otimes I} \left[\,  f (\omega) \neq f(\tilde\omega) \,\right].\end{align*}
Above, $\tilde\omega$ denotes the random element in $\Omega^I$ which is the same as $\omega$ in every coordinate except the $i$-th coordinate which is \emph{resampled independently}. 

\begin{remark}
  The \eqref{eq:OSSS} inequality is originally stated for the case when the sets $\Omega$ and $I$ are finite. However, the proof of \cite{OSSS} carries on for the case where $(\Omega,\pi)$ a general probability space and $I$ infinite without any need for modification. The reader could also consult \cite[Theorem 2.5]{DumRaoTas16a}.
\end{remark}

\subsection{Tensorization of Voronoi percolation}

We will eventually apply \eqref{eq:OSSS}. 
In order to do so, we introduce a suitable finite product space to encode the measure of Voronoi percolation.

Fix $\varepsilon >0$. For $x\in \varepsilon \mathbb Z^d$, introduce the box
$\mathsf R_{x}^\varepsilon:= x+[0,\varepsilon)^d$ as well as $\eta_x^b = \eta^b \cap \mathsf R_x^{\varepsilon}$, $\eta_x^w = \eta^{w} \cap \mathsf R_x^{\varepsilon}$ and $\eta_x = \eta_x^w \cup \eta_x^b$. 
Let $(\Omega_x, \pi_x)$ be the measured space associated to the random variable $\eta_x=(\eta^b_x, \eta^w_x)$, and consider the product space $(\prod_{x\in \eps \mathbb Z^d} \Omega_x, \bigotimes_{x\in \eps \mathbb Z^d} \pi_x)$. 
Since the random variables $(\eta^b_x, \eta^w_x)$ are independent for different $x$,  this space is in direct correspondence with the original space on which Voronoi percolation was defined. 

For $x\in\eps\mathbb Z^d$ and an increasing event $A$, define
\begin{equation}
  \label{eq:1}
   \infl_x^\eps[A]:=\mathbb P_p[\mathbf 1_A(\eta)\neq\mathbf 1_A (\tilde\eta)],
\end{equation}
where $\eta=(\eta_z)_{z\in\eps\mathbb Z^d}$ has law $\bigotimes_{z\in \eps \mathbb Z^d} \pi_z$ and  $\tilde\eta$ is equal to $\eta$ except on the $x$-coordinate which is resampled independently. Here and below, we use a slight abuse of notation by denoting the measure on the probability space in which $\eta$ and $\tilde\eta$ are defined by $\bbP_p$.
\begin{lemma}
  \label{lem:RussoInfl} For a local  increasing event $A$,
    \begin{equation}
    \label{eq:2}
      \frac{{\rm d}\bbP_p [A]}{{\rm d}p}  \ge \frac12 \, \limsup_{\ep\to 0}\sum_{x\in \ep \mathbb Z^d}  \infl_x^\eps[A].
  \end{equation}
\end{lemma}
\begin{proof}
Assume $A$ depends on the colors in $\mathsf B_n$ only. Let us start by proving that for any $m\ge 1$,
\begin{equation}
 \frac{{\rm d}\bbP_p [A]}{{\rm d}p} \ge\frac12 \, \limsup_{\ep\to 0}\sum_{x\in \eps \bbZ^d \cap \mathsf B_m } \infl_x^\eps[A].\label{eq:4}
\end{equation}
Fix $x\in \eps\mathbb Z^d\cap \mathsf B_m$ and use the notation for $\eta$ and $\tilde\eta$ introduced above. 
Observe that with probability $1-O(\varepsilon^{2d})$, there $\eta_x\cup \tilde\eta_x$ contains at most one point. Then, using that  $\eta=\tilde\eta$ when $\eta_x=\tilde\eta_x=\emptyset$ and that $\eta_x$ and $\tilde\eta_x$ play symmetric roles, we obtain that
  \begin{align*}
     \infl_x^\eps[A]&= 2\bbP_p[\mathbf 1_A(\eta)\neq \mathbf 1_A (\tilde\eta),|\eta_x|=1,|\tilde\eta_x|=0] +O(\eps^{2d}).
  \end{align*}
Under the condition that $|\eta_x|=1$ and $|\tilde\eta_x|=0$, the configuration $\tilde\eta$ is simply obtained from $\eta$ by removing the only point $x$ of $\eta$ in $\mathsf R_x^\eps$. Furthermore, by monotonicity, this point $x$ must be pivotal in $\eta$  when $\mathbf 1_A(\eta)\neq \mathbf 1_A (\tilde\eta)$. Hence, writing $\mathsf{Piv}_A$ for the pivotal set corresponding to $\eta$, the equation above implies
\begin{align*}
   \infl_x^\eps[A]&\le 2  \bbP_p[ |\mathsf{Piv}_A\cap \mathsf R_x^\eps|\ge 1, |\eta_x|=1,|\tilde\eta_x|=0]+O(\eps^{2d})\\
&\le  2  \bbE_p[| \mathsf{Piv}_A\cap \mathsf R_x^\eps| ]+O(\eps^{2d}).
\end{align*}
Summing this equation over the points $x\in\eps\bbZ^d\cap\mathsf B_m$ gives
 \begin{equation*}
   \sum_{x\in \eps \bbZ^d \cap \mathsf B_m}  \infl_x^\eps[A]\le 2 \bbE_p[|\mathsf{Piv}_A|]+ O(\eps^d).
 \end{equation*}
Eq. \eqref{eq:4} follows by taking the $\limsup$ and using the derivative formula of Lemma~\ref{Russo}.

Obtaining \eqref{eq:2} from \eqref{eq:4} follows readily from the existence of $c>0$ such that
\begin{equation}
  \label{eq:3}
   \infl_x^\eps[A]\le 2\eps^d  \exp(-c|x|^d)
\end{equation}
uniformly in $\eps$ and $x\in \eps\bbZ^d$ with $\|x\|\ge{4n}$.
To see this, assume that the value of $\mathbf 1_{A}$ is changed when $\eta$ is replaced by $ \tilde\eta$. Then, $\eta\cup\tilde\eta$ must have at least one points in $\mathsf R_x^\eps$ (which occurs with probability smaller than $2\eps^d$), and the cell of one of these points must intersect $\mathsf B_{\|x\|/4}$ (and therefore  $\mathsf B_{\|x\|/4}$ cannot contain a point of $\eta$).
\end{proof}
\section{Proof of Theorem~\ref{thm:main}}\label{sec:3}

As mentioned in the introduction, we only need to prove \eqref{eq:mlem}. For this, we fix $\delta>0$ with $p_c\in(\delta,1-\delta)$. Fix $n>0$ and $p\in[\delta,1-\delta]$. Below, constants $c_i$ ($i\le 4$) are positive and depend on $\delta$ and $d$ only. In particular, these constants are independent of $n$ and $p$.

For $\eps\in(0,1)$, consider the product space $(\prod_{x\in \eps \mathbb Z^d} \Omega_x, \bigotimes_{x\in \eps \mathbb Z^d} \pi_x)$ 
 introduced in Section 2.4. Applying \eqref{eq:OSSS} to $f=\mathbf 1_{0\longleftrightarrow\mathsf S_n}$ and an algorithm $\mathsf T_k$ determining $f$ gives that
\begin{equation}\label{eq:154}\theta_n(p)(1-\theta_n(p))~\le~\sum_{x\in \varepsilon\bbZ^d} \delta_x(\mathsf T_k) \infl_{x}^{\,\varepsilon}[0 \longleftrightarrow\mathsf S_n].\end{equation}

The algorithm $\mathsf T_k$ will be provided by the following  lemma, whose proof is postponed to the end of this section.\begin{lemma}\label{lem:11}
There exists $c_0>0$ such that for any $k\in\llbracket 1,n\rrbracket$, there exists an algorithm $\mathsf T_k$ determining $\mathbf 1_{0 \longleftrightarrow\mathsf S_n}$ with the property that 
\begin{equation}\label{eq:177}
\delta_x (\mathsf T_k) \leq c_0 \, \bbP_p[  x  \longleftrightarrow  \mathsf S_k].
\end{equation}
\end{lemma}
 Now,  using \eqref{eq:177} in \eqref{eq:154} gives \begin{equation}\label{eq:166}
\theta_n(p)~\le~c_0c_1\sum_{x \in \varepsilon\bbZ^d} \bbP_p[  x\longleftrightarrow\mathsf S_k]\,{\infl}^{\, \varepsilon}_{x}[0\longleftrightarrow \mathsf S_n]
\end{equation}
where $c_1:=(1-\theta_1(1-\delta))^{-1}$. Averaging \eqref{eq:166} over $1\le k\le n$  gives
\begin{equation*}
\theta_n(p)~\le~\frac{c_0c_1}n\sum_{x \in \varepsilon\bbZ^d}  \Big(\sum_{k=1}^n\, \bbP_p[  x\longleftrightarrow\mathsf S_k]\Big)\, {\infl}^{\, \varepsilon}_{x}[0\longleftrightarrow\mathsf S_n].
\end{equation*}
A simple geometric observation using the invariance under translation of Voronoi percolation implies that 
$$\sum_{k=1}^n\, \bbP_p[  x \longleftrightarrow  \mathsf S_k]~\le~ \sum_{k=1}^n\, \theta_{d(x,\mathsf S_k)}(p)~\le~ 2S_n(p)$$
(above, $d(x,\mathsf S_k)$ denotes the distance between $x$ and $\mathsf S_k$) so that
\begin{equation*}
\theta_n(p)~\le~2c_0c_1 \frac{S_n(p)}n\sum_{x \in \varepsilon\bbZ^d} {\infl}^{\, \varepsilon}_{x}[0\longleftrightarrow \mathsf S_n].
\end{equation*}
Lemma~\ref{lem:RussoInfl} implies \eqref{eq:mlem} by letting $\varepsilon$ tend to 0.
Overall, the proof of the theorem boils down to the proof of Lemma~\ref{lem:11}.
\begin{proof}[Proof of Lemma~\ref{lem:11}]
Fix $k\in\llbracket 1,n\rrbracket$. 
We start by defining the algorithm. 

For each $y\in \varepsilon\bbZ^d$, define an auxiliary algorithm $\mathsf {Discover}(y)$ revealing the random variables $\eta_x$ around the point $y$ until the color of each point in $\mathsf R_y^{\varepsilon}$ is determined. 
More formally, set $s=0$. 
When $s=t$, if the color of all the points inside the box $\mathsf R_y^{\varepsilon}$ is determined by all the revealed coordinates so far, the algorithm stops and returns the colors of points as the output. 
If not, the algorithm reveals the value of  $\eta_x$ for $x\in \eps\bbZ^d$ satisfying $\|x-y\|\le t$ and sets $s=t+1$. We write $x \in D(y)$ if $x$ is revealed by $\mathsf {Discover}(y)$.
We are now in a position to define the algorithm $\mathsf T_k$.

\begin{definition}
Set $X_0= \emptyset$ and $Z_0 =\mathsf S_k$. At step $t$, assume that $X_t \subset \bbZ^d$ and $Z_t \subset \bbR^d$ have been constructed.
If there is no $y \in \varepsilon\bbZ^d \setminus X_t$ with $\mathsf R_y^\varepsilon \cap Z_t \neq \emptyset$, the algorithm stops.
If such a $y$ exists (if more than one exists, pick the smallest for an ordering of $\varepsilon\bbZ^d$ fixed before running the algorithm), then the algorithm does the following:
\begin{itemize}[noitemsep,nolistsep]
\item[-] $\mathsf{Discover}(y)$.
\item[-] Set $X_{t+1}=X_t\cup \{ y \}$.
\item[-] Set $Z_{t+1} = Z_t \cup \{ \text{all the black points in $\omega\cap\mathsf S_y^\eps$} \}$.
\end{itemize}
\end{definition}

Note that this algorithm discovers the connected component of $\mathsf S_k$ in $\omega$. In particular, it clearly determines ${\bf 1}_{0 {\longleftrightarrow}\mathsf S_n}$.
We now bound the revealment of $\mathsf T_k$. 
 
When $x\in\eps\bbZ^d$ is revealed, there exist $y\in \eps\bbZ^d$ and $y'\in\mathsf R_y^\eps$ such that $x\in D(y)$ and $y' {\longleftrightarrow}\mathsf S_k$. This $y'$ belongs to $\mathsf R_z^1(=z+[0,1)^d)$ for some $z\in\bbZ^d$. Note that in this case, the fact that $x\in D(y)$ implies in particular that $\eta^b$ does not intersect the Euclidean ball of radius $\|x-z\|-3\sqrt d$ around $z$ since otherwise the color of any point in $\mathsf R_z^1$ is independent of the colors of points in $\mathsf R_x^\varepsilon$. Let $E_z$ be this last event (which is decreasing). We find
\begin{align*}
 \delta_x(\mathsf T_k) 
&\stackrel{\phantom{\rm FKG}}\le
\sum_{z\in\bbZ^d} \bbP_p [\mathsf R_z^1 \longleftrightarrow \mathsf S_k~,~E_z]\stackrel{\rm FKG}\le
\sum_{z\in\bbZ^d} \bbP_p [\mathsf R_z^1 \longleftrightarrow \mathsf S_k]\cdot\bbP_p[E_z].
\end{align*}
A standard estimate on Poisson Point Processes in $\bbR^d$ implies that 
\begin{equation}\label{eq:cde}\bbP_p[E_z]\le \tfrac{1}{c_2}\exp(-c_2 \|z-x\|^d).\end{equation} Furthermore, when $z\in \mathsf B_m$, by choosing a path  $y_0,\dots,y_k=z$ in $\bbZ^d$ with $x\in \mathsf R_{y_0}^1$ and $k\le c_3\|z-x\|$, we deduce that 
\begin{align}\bbP_p[x  \longleftrightarrow  \mathsf S_k \vert  \mathsf R_z^1  \longleftrightarrow   \mathsf S_k ]&\stackrel{\rm FKG}\ge\bbP_p[x \longleftrightarrow \mathsf R_{z}^1\ ,\ \mathsf R_{z}^1\,\text{all black}]\nonumber\\
& \stackrel{\rm FKG}\geq \prod_{i=1}^k\bbP_p[\mathsf R_{y_i}^1\,\text{all black}]\nonumber\\
& \stackrel{\phantom{\rm FKG}}\ge \exp(-c_4\|z-x\|)\label{eq:cdef} .\end{align}
(In the last inequality we used that $p\ge\delta$.) The bounds \eqref{eq:cde} and \eqref{eq:cdef} imply that
\begin{align*}
  \delta_x(\mathsf T_k) &\le \bbP_p [x \longleftrightarrow \mathsf S_k]\sum_{z\in\bbZ^d}\exp(c_4\|z-x\|)\cdot \tfrac1{c_2}\exp(-c_2\|z-x\|^d)\le  c_0  \bbP_p [x\longleftrightarrow \mathsf S_k],
\end{align*}
which concludes the proof.
\end{proof}

\section{Proof of Corollary~\ref{cor:main}}\label{sec:2}


Let $A_n$ be the event that $\Lambda_n:=[-n,n]^2$ is crossed by a continuous path of black points going from left to right. Since the complement of $A_n$ is the event that there is a continuous path of white vertices from top to bottom, which has the same probability, we deduce that 
\begin{equation}\bbP_{1/2}[A_n]=1/2.\label{eq:yy}\end{equation}

In particular, \eqref{eq:yy} implies that $\bbP_{1/2}[\mathsf B_1\longleftrightarrow \mathsf S_n]\ge 1/n$ so that 
$$\bbP_{1/2}[0\longleftrightarrow \mathsf S_n]\stackrel{\rm FKG}\ge \bbP_{1/2}[\mathsf B_1\longleftrightarrow \mathsf S_n]\bbP_{1/2}[\mathsf B_1\text{ all black}]\ge \tfrac{1}n\bbP_{1/2}[\mathsf B_1\text{ all black}].$$
Since this quantity does not decay exponentially fast, we deduce that $p_c\le 1/2$.

The square-root trick (using the FKG inequality) implies that for any $n\ge k\ge1$,
$$\bbP_{1/2}[\mathsf B_k\text{ is connected in $\Lambda_n$ to the top of }\Lambda_n]\ge 1-\bbP_{1/2}[\mathsf B_k\not\longleftrightarrow \infty]^{1/4}$$
 so that 
$$\bbP_{1/2}[\mathsf B_k\text{ is connected in $\Lambda_n$ to the top and bottom of }\Lambda_n]\ge 1-2\bbP_{1/2}[\mathsf B_k\not\longleftrightarrow \infty]^{1/4}.$$
Now, the uniqueness of the infinite connected component \cite{BolRio06b} when it exists implies that 
$$\liminf_{n\rightarrow\infty}\bbP_{1/2}[A_n]\ge 1-2\bbP_{1/2}[\mathsf B_k\not\longleftrightarrow \infty]^{1/4}.$$
Assume for a moment that $\theta(1/2)>0$. Letting $k$ tend to infinity, we would deduce that $\bbP_{1/2}[A_n]$ tends to 1 which would contradict \eqref{eq:yy}. This implies $\theta(1/2)=0$ and $p_c\ge1/2$.

\section{Proof of Lemma~\ref{lem:technical}}

Define 
$\displaystyle \beta_1 := \inf \big\{ \beta \, : \, \limsup_{n \rightarrow \infty} \frac{\log \Sigma_n(\beta)}{\log n} \geq 1 \big\}. $

\paragraph{Assume $\beta<\beta_1$.} Fix $\delta>0$ and set $\beta'=\beta-\delta$ and $\beta''=\beta-2\delta$. We will prove that there is exponential decay at $\beta''$ in two steps. 

First, there exists an integer $N$ and $\alpha>0$  such that $\Sigma_n(\beta) \leq n ^ {1-\alpha}$ for all $n \geq N$. For such an integer $n$, integrating $f_n' \geq n^{\alpha} f_n$  between $\beta'$ and $\beta$ -- this differential inequality follows from \eqref{eq:mlem}, the monotonicity of the functions $f_n$ (and therefore $\Sigma_n$) and the previous bound on $\Sigma_n(\beta)$ -- implies that
$$ f_n(\beta') \leq M\exp(-\delta \,n^\alpha),\quad\forall n\ge N.$$

Second, this implies that there exists $\Sigma < \infty$ such that $\Sigma_n(\beta') \leq \Sigma$ for all $n$. Integrating $f_n' \geq \tfrac{n}{\Sigma} f_n$ for all $n$ between $\beta''$ and $\beta'$ -- this differential inequality is again due to \eqref{eq:mlem}, the monotonicity of $\Sigma_n$, and the bound on $\Sigma_n(\beta')$ -- leads to
\begin{equation*}f_n (\beta'') \leq M\exp(-\frac{\delta}{\Sigma} \,n),\quad\forall n\ge0.\label{eq:bb}\end{equation*}

\paragraph{Assume $\beta>\beta_1$.} For $n\ge 1$, define the function $T_n := \frac{1}{\log n} \sum_{i=1}^{n} \frac{f_i}{i}$. Differentiating $T_n$ and using \eqref{eq:mlem}, we obtain
$$T_n'~=~ \frac{1}{\log n}\, \sum_{i=1}^{n} \frac{f_i'}{i}  ~\stackrel{\eqref{eq:mlem}}{\geq}~ \frac{1}{\log n}\, \sum_{i=1}^{n} \frac{f_i}{\Sigma_{i}} ~\geq~ \frac{\log \Sigma_{n+1}-\log \Sigma_1}{\log n},$$
where in the last inequality we used that for every $i\ge1$,
$$\frac{f_i}{\Sigma_{i}} \geq \int_{\Sigma_{i}}^{\Sigma_{i+1}} \frac{dt}{t}=\log \Sigma_{i+1}-\log \Sigma_{i}.$$ 
For $\beta'\in(\beta_1,\beta)$, using that $\Sigma_{n+1}\ge\Sigma_n$ is increasing and integrating the previous differential inequality between $\beta'$ and $\beta$ gives
$$T_n (\beta) - T_n(\beta') \geq (\beta - \beta')\,  \frac{\log \Sigma_{n} (\beta')-\log M}{\log n}.$$
Hence, the fact that $T_n(\beta)$ converges to $f(\beta)$ as $n$ tends to infinity implies 
\begin{equation*} f (\beta) -f(\beta') 
~\geq~  (\beta - \beta')  \, \big[\limsup_{n\rightarrow \infty} \frac{\log \Sigma_n (\beta')}{\log n}\big]~\geq~ \beta- \beta'.
\end{equation*}
Letting $\beta'$ tend to $\beta_1$ from above, we obtain
$
f(\beta)\ge \beta-\beta_1.
$

\paragraph{Acknowledgments} 
The authors are thankful to Asaf Nachmias for reading the manuscript and his helpful comments. 
This research was supported by the IDEX grant from Paris-Saclay, a grant from the Swiss FNS, and the NCCR SwissMAP.

\bibliographystyle{alpha}
\bibliography{biblicomplete}
\end{document}